\documentclass[11pt,amssymb]{amsart}

\usepackage{txfonts}

\usepackage{mathtools} %for using "dcase" instead of "case" to display style!

\usepackage{amsmath}
\usepackage{amssymb}
\usepackage{amsthm}
\usepackage{newlfont}
\usepackage{graphicx}
\newtheorem{thm}{Theorem}[section]

\newtheorem{lem}[thm]{Lemma}

\theoremstyle{definition}

\theoremstyle{remark}
\newtheorem{rem}[thm]{Remark}
\numberwithin{equation}{section}
\newcommand{\bes}[1]{\dot B^{#1}_{\infty,\infty}}

\newcommand{\ed}{\end {document}}
\title[Besov regularity]
{Small $\dot B^{-1}_{\infty,\infty}$ implies regularity}

\author[T. Hmidi]{Taoufik Hmidi}
\address[T. Hmidi]{IRMAR, Universit\'e de Rennes 1\\ Campus de
Beaulieu\\ 35~042 Rennes cedex\\ France}
\email{thmidi@univ-rennes1.fr}

\author[D. Li]{Dong Li}
\address[D. Li]{Department of Mathematics, University of British Columbia, Vancouver BC Canada V6T 1Z2}%
\email{dli@math.ubc.ca}

\begin{document}
\begin{abstract}
We show that smallness of $\dot B^{-1}_{\infty,\infty}$ norm of solution to $d$-dimensional
($d\ge 3$) incompressible
Navier-Stokes prevents blowups.
\end{abstract}
\maketitle
%--------------------------------------------------------------------
\section{Introduction}
%--------------------------------------------------------------------
In recent \cite{FGL16}, Farhat, Gruji\'c and Leitmeyer proved that any unique $L^{\infty}$
mild solution to 3D Navier-Stokes equation cannot develop finite-time blowups if the
$B^{-1}_{\infty,\infty}$ norm is sufficiently small (near first possible blowup time). This
result is perhaps a bit surprising in view of the  illposedness result of 
Bourgain-Pavlovi\'c \cite{BP08}. The proof in \cite{FGL16} has a strong geometric flavor, 
and in particular relies on a geometric regularity
criteria and characterization of the super-level sets developed in the series of works 
\cite{DaGr12-3,Gr13,GrKu98}. We refer the readers to the introduction in \cite{FGL16} and
the references therein (see also \cite{AG13}--\cite{Le34})
for more details on these techniques and also related developments. 
The purpose of this note is to revisit this problem from
the point of view of Littlewood-Paley calculus. In particular we will give a streamlined 
proof for all dimensions $d\ge 3$.

Consider $d$-dimensional Navier-Stokes Equation (NSE):
\begin{align} \label{e1}
\begin{cases}
\partial_t v  + (v \cdot \nabla) v = \Delta v - \nabla p, 
\quad (t,x) \in (0,\infty)\times \mathbb R^d, \\
\nabla \cdot v =0, \\
v\Bigr|_{t=0}=v_0.
\end{cases}
\end{align}

\begin{thm}\label{thm1}
Let $d\ge 3$. Suppose $v$ is a smooth solution to \eqref{e1} and let $T>0$ be the first possible blow-up time. 
There exists a positive constant $m_0$ depending only on the dimension $d$,
 such that if the solution $v$
satisfies
\[ 
 \sup_{t \in (T-\epsilon, T)} \|v(t)\|_{\dot B^{-1}_{\infty, \infty}} \le m_0,
\]
for some $0 < \epsilon <T$, then $T$ is not a blow-up time, and the solution
can be continued past $T$.
\end{thm}

\begin{rem}
Here to allow some generality we do not specify the particular class of smooth solution.
As an example one can consider as in \cite{FGL16} the unique mild solution emanating
from $L^{\infty}$ initial data. By smoothing (cf. \cite{DLCMS}) the solution is immediately in $W^{k,\infty}$
for all $k$. Other classes of solutions can also be considered and we will not dwell
on this issue here.

\end{rem}

We gather below some notation used in this note.
\subsection*{Notation}
For any two quantities $X$ and $Y$, we denote $X \lesssim Y$ if
$X \le C Y$ for some constant $C>0$. The dependence of the constant $C$ on
other parameters or constants are usually clear from the context and
we will often suppress  this dependence.

We will need to use the Littlewood--Paley (LP) frequency projection
operators. To fix the notation, let $\phi_0 \in
C_c^\infty(\mathbb{R}^n )$ and satisfy
\begin{equation}\nonumber
0 \leq \phi_0 \leq 1,\quad \phi_0(\xi) = 1\ {\text{ for}}\ |\xi| \leq
1,\quad \phi_0(\xi) = 0\ {\text{ for}}\ |\xi| \geq 7/6.
\end{equation}
Let $\phi(\xi):= \phi_0(\xi) - \phi_0(2\xi)$ which is supported in $\frac 12 \le |\xi| \le \frac 76$.
For any $f \in \mathcal S(\mathbb R^n)$, $j \in \mathbb Z$, define
\begin{align*}
 &\widehat{P_{\le j} f} (\xi) = \phi_0(2^{-j} \xi) \hat f(\xi), \\
 &\widehat{P_j f} (\xi) = \phi(2^{-j} \xi) \hat f(\xi), \qquad \xi \in \mathbb R^n.
\end{align*}
Sometimes for simplicity we write $f_j = P_j f$, $f_{\le j} = P_{\le j} f$. Note that by using the support property of $\phi$, we have $P_j P_{j^{\prime}} =0$ whenever $|j-j^{\prime}|>1$.
The Bony paraproduct for a pair of functions
$f,g$ take the form
\begin{align*}
f g = \sum_{i \in \mathbb Z} f_i \tilde g_i + \sum_{i \in \mathbb Z} f_i g_{\le i-2} + \sum_{i \in \mathbb Z}
g_i f_{\le i-2},
\end{align*}
where $\tilde g_i = g_{i-1} +g_i + g_{i+1}$.
For $s\in \mathbb R$, $1\le p\le \infty$, the homogeneous Besov $\dot B^s_{\infty,\infty}$ norm is given by
\begin{align*}
 \| f \|_{\dot B^s_{\infty,\infty} } 
 = \sup_{j\in \mathbb Z}  \bigl( 2^{js} \| P_j f \|_{\infty} \bigr).
\end{align*}
We will use without explicit mentioning the simple estimate:
\begin{align*}
\| e^{t \Delta} P_j f \|_{L^{\infty}(\mathbb R^d)}
 \lesssim e^{-c 2^{2j}t} \| P_j f\|_{L^{\infty}(\mathbb R^d)},\quad \forall\, t>0,
\end{align*}
where $c>0$ is a constant depending only on $d$.

%and
%$f_{a\le\cdot\le b} = \sum_{j=a}^b f_j$.

\section{Proof of Theorem \ref{thm1}}

\begin{lem} \label{lem1}
Let $\gamma>1$. Then for any $j \in \mathbb Z$, we have
\begin{align*}
\| P_j ( ( v\cdot \nabla) v ) \|_{\infty}
\lesssim 2^{j(2-\gamma)} \| v \|_{\bes{-1}} \| v \|_{\bes{\gamma}}.
\end{align*}
\end{lem}

\begin{proof}[Proof of Lemma \ref{lem1}]
Although this is utterly standard we give a proof for completeness.
 By paraproduct decomposition,
we have
\begin{align*}
(v\cdot \nabla) v 
& = \sum_{l\in \mathbb Z} (v_{\le l-2} \cdot \nabla) v_l 
+ \sum_{l \in \mathbb Z} (v_l \cdot \nabla) v_{\le l-2}
+ \sum_{l \in \mathbb Z} (v_l \cdot \nabla) \tilde v_l \notag \\
&=: \; A+ B+C,
\end{align*}
where $\tilde v_l = v_{l-1} + v_l + v_{l+1}$. 
Then by frequency localization,  we have
\begin{align*}
\| P_j (A) \|_{\infty}
\lesssim \sum_{|l-j|\le 2} \| v_{\le l-2} \cdot \nabla v_l \|_{\infty} 
\lesssim 2^j \| v\|_{\bes{-1}} \cdot 2^{j(1-\gamma)} \| v \|_{\bes{\gamma}}.
\end{align*}
Similar estimate hold for $B$. Now for the estimate of $C$, note that by using divergence-free
property we can write $(v_l \cdot \nabla)\tilde v_l = \nabla \cdot( v_l \otimes \tilde v_l)$
and this gives
\begin{align*}
\|P_j (C)\|_{\infty}
\lesssim 2^j \sum_{l\ge j-2} 2^{-l} \|v_l\|_{\infty}
\cdot \| \tilde v_l \|_{\infty} \cdot 2^{\gamma l} \cdot 2^{-l(\gamma-1)} 
\lesssim 2^{j(2-\gamma)} \| v \|_{\bes{-1}} \| v\|_{\bes{\gamma}}.
\end{align*}
Here we used the assumption $\gamma>1$.

\end{proof}

\begin{lem} \label{lem2}
Suppose $v=v(t)$ is a smooth solution to \eqref{e1} on some time interval
$[0,T]$ with smooth initial data $v_0$. Let $\gamma>1$. There exists constants
$C_1>0$, $\delta_1>0$ which depend only on $(\gamma,d)$, such that
if 
\begin{align*}
\sup_{0\le t\le T} \| v (t) \|_{\bes{-1}} \le \delta_1,
\end{align*}
then
\begin{align*}
\max_{0\le t \le T} \| v(t) \|_{\bes{\gamma}} \le C_1 \| v_0 \|_{\bes{\gamma}}.
\end{align*}
\end{lem}
\begin{proof}[Proof of Lemma \ref{lem2}]
Write $v_j = P_j v$. Then 
\begin{align*}
\partial_t v_j -\Delta v_j =- P_j \bigl( \Pi ( (v\cdot \nabla )v ) \bigr),
\end{align*}
where $\Pi$ is the usual Leray projection operator. Then for any $t>0$, by using
Lemma \ref{lem1}, we have
\begin{align*}
\|v_j (t)\|_{\infty}
& \lesssim e^{-c 2^{2j} t}
\| v_j (0) \|_{\infty}
+ \int_0^t e^{-c \cdot 2^{2j} (t-s)}
2^{j(2-\gamma)} \| v(s) \|_{\bes{-1}} \| v(s) \|_{\bes{\gamma}} ds \notag \\
& \lesssim e^{-c 2^{2j} t}
\|v_j(0)\|_{\infty}
+ (1-e^{-c 2^{2j} t}) \cdot 2^{-j\gamma}
\cdot \sup_{0\le s\le t} \| v(s) \|_{\bes{-1}} \cdot
\max_{0\le s \le t} \|v(s) \|_{\bes{\gamma}}.
\end{align*}
This implies that for some constants $\tilde C_1>0$, $\tilde C_2>0$ depending only
on $(\gamma,d)$, 
\begin{align*}
\max_{0\le t \le T}
\| v(t) \|_{\bes{\gamma}}
\le \tilde C_1 \| v_0 \|_{\bes{\gamma}} 
+\tilde C_2 \cdot \sup_{0\le t\le T} \| v(t) \|_{\bes{-1}}
\cdot \max_{0\le t\le T} \| v(t) \|_{\bes{\gamma}}.
\end{align*}
The result obviously follows.
\end{proof}

\begin{proof}[Proof of Theorem \ref{thm1}]
Choose $\gamma=3/2$ and $m_0=\delta_1$ as specified in Lemma \ref{lem2}. Consider
the solution $v=v(t)$ on the time interval $[T-{\epsilon}, T-\eta]$, where $\eta>0$ will
tend to zero. By Lemma \ref{lem2} (regarding $v(T-\epsilon)$ as initial data), we then
obtain uniform estimate on $ \|v\|_{\bes {\gamma}}$ independent of $\eta$. A standard
argument then implies that $v$ must be regular beyond $T$.

\end{proof}

%-----------------------------------------------------------------------
\section*{Acknowledgements}
D. Li was supported by an Nserc grant.
T. Hmidi  was partially supported by the ANR project Dyficolti ANR-13-BS01-0003- 01. 
%-----------------------------------------------------------------------

%---------------------------------------------------------------------

\begin{thebibliography} {000000000}

\bibitem{AG13}
	K. Abe and Y. Giga, Y,
	{Analyticity of the Stokes semigroup in spaces of bounded functions},
	\emph{Acta Math}, \textbf{211, no. 1} (2013), 1--46.
    
\bibitem{AG14}
    K. Abe and Y. Giga, 
    {The $L^\infty$-Stokes semigroup in exterior domains},
    \emph{J.\ Evol.\ Equ}, \textbf{14, no. 1} (2014), 1--28.
    
\bibitem{BP08}
J. Bourgain and N. Pavlovi\'c,
{Ill-posedness of the Navier-Stokes equations in a critical space in 3D}.
\emph{J. Funct. Anal}, \textbf{255} (2008), 2233--2247.


\bibitem{ChSh10}
A. Cheskidov and R. Shvydkoy,
{The regularity of weak solutions of the 3D Navier-Stokes equations in $B^{-1}_{\infty, \infty}$},
\emph{Arch. Ration. Mech. Anal.}, \textbf{195} (2010), 159--169.

\bibitem{CPS95}
P. Constantin, I. Procaccia and D. Segel, 
{Creation and dynamics of vortex tubes in three dimensional turbulence},
\emph{Phys. Rev E} \textbf{51} (1995),
3207.

\bibitem{DaGr12-3}
R. Dascaliuc and Z. Gruji\'c, 
{Vortex stretching and criticality for the 3D NSE},
\emph{J. Math. Phys.} \textbf{53, no. 11} (2012), 115613, 9 pp.

\bibitem{DLCMS}
H. Dong and D. Li,
Optimal local smoothing and analyticity rate estimates for the generalized Navier-Stokes equations,
\emph{Commun. Math. Sci.}, \textbf{7, no. 1} (2009),  67--80. 

\bibitem{ESS03}
L. Escauriaza, G. Seregin and V. Shverak, 
{$L_{3,\infty}$-solutions of Navier-Stokes equations and backward uniqueness},
\emph{Uspekhi Mat. Nauk.}
\textbf{58} (2003), 211--250.

\bibitem{FGL16}
A. Farhat, Z. Z. Gruji\'c and K. Leitmeyer.
{The space $B^{-1}_{\infty,\infty}$, volumetric
sparseness, and $3D$ NSE}. \emph{Preprint}. arXiv:1603.08763v2


\bibitem{GrKu98}
Z. Gruji\'c and I. Kukavica, 
{Space analyticity for the Navier-Stokes and related equations with initial data in $L^p$},
\emph{J. Funct. Anal.} \textbf{152} (1998), 447--466.

\bibitem{Gr13}
Z. Gruji\'c,
{A geometric measure-type regularity criterion for
solutions to the 3D Navier-Stokes equations}, \emph{Nonlinearity} \textbf{26}
(2013), 289--296.

\bibitem{Gu10}
R. Guberovi\'c, 
{Smoothness of Koch-Tataru solutions to the Navier-Stokes equations
revisited},
\emph{Discrete Cont. Dynamical Systems} \textbf{27, no. 1} (2010), 
231--236.



%\bibitem{KT01}
%H. Koch and D. Tataru,
%\emph{Well posedness for the Navier-Stokes equations},
%Adv. Math. \textbf{157}, 22 (2001).

\bibitem{Le34}
J. Leray, 
{Sur le mouvement d'un liquide visqueux emplissant l'espace}, \emph{Acta Math}. \textbf{63, no. 1} (1934), 193--248.













\end{thebibliography}
\end{document}